\documentclass[11pt, oneside, a4paper]{amsart}
\usepackage[a4paper]{geometry}
\usepackage[colorlinks=true ]{hyperref}
\usepackage[english]{babel}
\usepackage[utf8]{inputenc}
\usepackage{amsthm,amsmath,amssymb}
\usepackage{color,graphicx}
\usepackage[backend=bibtex, style=numeric, isbn=false, url=false]{biblatex}
\bibliography{Vlasovb}

\newcounter{mnotecount}[section]

\numberwithin{equation}{section}
\newtheorem{theorem}{Theorem}
\newtheorem{definition}[theorem]{Definition}
\newtheorem{lemma}[theorem]{Lemma}

\newtheorem{remark}[theorem]{Remark}
\newtheorem{corollary}[theorem]{Corollary}

\newcommand{\Naturals}{\mathbb{N}}
\newcommand{\Reals}{\mathbb{R}}

\newcommand{\Stwo}{S^2}

\newcommand{\di}{\mathrm{d}}
\newcommand{\diNormal}{\di\nu}
\newcommand{\diVol}{\di\mu_g}
\newcommand{\diVolSigmat}{\di\mu_{\Sigma_t}}
\newcommand{\diVolCone}{\di\mu_{\fdconex}}
\newcommand{\diVolModelx}{\di^3 {x}}
\newcommand{\diVolModeltx}{\di^4 {x}}
\newcommand{\diVolModelv}{\di^3 v}

\newcommand{\pt}{\partial_t}
\newcommand{\pr}{\partial_r}
\newcommand{\ph}{\partial_\theta}
\newcommand{\pp}{\partial_\phi}

\newcommand{\Lie}{\mathrm{Lie}}
\newcommand{\metricg}{g}

\newcommand{\ia}{a}
\newcommand{\ib}{b}
\newcommand{\ic}{c}
\newcommand{\id}{d}
\newcommand{\ie}{e}
\newcommand{\ig}{g}

\newcommand{\ua}{{\underline{a}}}
\newcommand{\ub}{{\underline{b}}}
\newcommand{\uc}{{\underline{c}}}

\newcommand{\ops}{(}
\newcommand{\cls}{)}

\newcommand{\EnergyModel}{E_{\textrm{model,3}}}
\newcommand{\chiFar}{\mathbf{1}_{r\not\simeq3M}}

\newcommand{\GeodesicEnergy}{e}
\newcommand{\GeodesicLz}{l_z}
\newcommand{\GeodesicQ}{q}

\newcommand{\fdcone}{{\mathcal{C}^+}}
\newcommand{\fdconex}{{\mathcal{C}^+_x}}
\newcommand{\ptx}{x}  % A point, x. 
\newcommand{\xb}{x}
\newcommand{\vb}{v}
\newcommand{\tb}{t}
\newcommand{\rb}{r}
\newcommand{\hb}{\theta}
\newcommand{\pb}{\phi}

\newcommand{\KDelta}{\Delta}
\newcommand{\KSigma}{\Sigma}
\newcommand{\KPi}{\Pi}
\newcommand{\CurlyR}{\mathcal{R}}

\newcommand{\DiffCurlyRTilde}{\tilde{\mathcal{R}}'}

\newcommand{\DDiffCurlyRTTilde}{\tilde{\tilde{\mathcal{R}}}''}
\newcommand{\OpL}{\mathcal{L}}
\newcommand{\TensorQ}{Q}

\newcommand{\EMS}{\mathrm{T}}
\newcommand{\SymOp}{S}
\newcommand{\GenEnergy}[1]{E_{#1}}
\newcommand{\GenBulk}[1]{\Pi_{#1}}

\newcommand{\vecTperp}{T_\perp}
\newcommand{\omegaperp}{\omega_\perp}
\newcommand{\vecTBlend}{T_{\chi}}
\newcommand{\omegaH}{\omega_{\mathcal{H}}}
\newcommand{\vecTBlendBF}{\mathbf{T}_{\chi}}
\newcommand{\vecMorawetz}{\mathbf{A}}
\newcommand{\fnq}{q}
\newcommand{\fnMa}{z}
\newcommand{\fnMb}{w}
\newcommand{\fnMaa}{z_1}
\newcommand{\fnMab}{z_2}
\newcommand{\fnMba}{w_1}
\newcommand{\fnMbb}{w_2}

\newcommand{\epsilondtsquared}{\epsilon_{e^2}}

\newcommand{\half}{\tfrac{1}{2}}

\title[Hidden symmetries and Kerr-Vlasov decay]{Hidden symmetries and decay for the Vlasov equation on the Kerr spacetime}
\author{L. Andersson}
\address{Albert Einstein Institute (Max-Planck Institute for Gravitational Physics)\\
Am Mühlenberg 1\\
D-14476 Potsdam, Germany }
\email{laan@aei.mpg.de}
\author{P. Blue}
\address{School of Mathematics\\
The University of Edinburgh\\
Maxwell Institute for Mathematical Sciences\\
James Clerk Maxwell Building\\
Peter Guthrie Tait Road\\
Edinburgh\\
EH9 3FD\\
United Kingdom}
\email{P.Blue@ed.ac.uk}
\author{J. Joudioux}
\address{Gravitationsphysik\\
Faculty of Physics\\
University of Vienna\\
Währinger Strasse, 17\\
1090 Vienna\\
Austria}
\email{jeremie.joudioux@univie.ac.at}
\date{\today}

\begin{document}
\begin{abstract}
This paper proves the existence of a bounded energy and integrated
energy decay for solutions of the massless Vlasov equation in the
exterior of a very slowly rotating Kerr spacetime. This combines methods previously developed to prove similar results for the wave equation on the exterior of a very slowly rotating Kerr spacetime with recent work applying the vector-field method to the relativistic Vlasov equation. 
\end{abstract}

\maketitle

%\tableofcontents 
\section{Introduction}

In this paper we prove the existence of a bounded energy, and an integrated energy decay estimate for
solutions of the massless Vlasov equation in the exterior of a very
slowly rotating Kerr spacetime. \footnotetext{\bfseries Preprint number \mdseries: UWThPh-2016-28 (University of Vienna)}

For parameters $a, M$, with $|a| \leq M$, the exterior region of the Kerr spacetime is represented in Boyer-Lindquist coordinates $(t,r,\theta,\phi)$ by $\Reals\times (r_+,\infty)\times\Stwo$ with the Lorentzian metric
\begin{align}
\metricg
= -\left(1-\frac{2Mr}{\KSigma}\right) \di t^2 
-\frac{4Mar\sin^2\theta}{\KSigma}\di t\di\phi 
 +\frac{\KPi\sin^2\theta}{\KSigma}\di\phi^2
+\frac{\KSigma}{\KDelta}\di r^2
+\KSigma \di\theta^2 ,
\label{eq:KerrMetricIntrouctureucture}
\end{align}
where $r_+=M+\sqrt{M^2-a^2}$, and
\begin{align*}
\KDelta&=r^2-2Mr+a^2, &
\KSigma&=r^2+a^2\cos^2\theta ,&
\KPi&=(r^2+a^2)^2-a^2\KDelta\sin^2\theta .
\end{align*}
For $|a|\leq M$, the Kerr spacetimes contain a black hole and are stationary and axisymmetric, that is to say $\partial_t$ and $\partial_\phi$ are Killing vector fields. Although the exterior is extendible as an analytic manifold, it is globally hyperbolic and foliated by surfaces of constant $t$, $\Sigma_t$, which are Cauchy surfaces. 

The Vlasov equation governs the evolution of massive or massless particles which do not self-interact \cite{lindquist:1966}. The particles are represented by a distribution function on phase space, which evolves under the geodesic flow, so it is constant along geodesics. In the context of kinetic theory, the equation is known as the collisionless Boltzmann equation. 

Let $(\mathcal{M}, \metricg)$ be a time oriented Lorentzian manifold 
%$(\mathcal{M},\metricg)$ 
of dimension $1+3$, with timelike vector field $T_+$. For the case of massless Vlasov, the distribution function is a non-negative function defined on the bundle of future 
light cones 
$\fdcone$, 
\begin{align*}
\fdcone &= \bigcup_{x\in \mathcal{M}} \fdconex, \\
\fdconex &= \{ (x,v): v\in T_p\mathcal{M}, \ g(v,v)=0, g(v,T_+)< 0 \} .
\end{align*}
For $m > 0$, the set $g(v,v) = -m^2$, $g(v,T_+) < 0$ is sometimes called the mass shell, and $\fdcone$ is its analogue for the massless case considered here.
The Vlasov equation is
\begin{equation}\label{eq:Vlasov}
\mathcal{X} f = 0,
\end{equation}
where $\mathcal{X}$ is the geodesic spray, the vector field on $T\mathcal{M}$ which generates the geodesic flow. The geodesic spray is the Lagrangian vector field of $L = \half g(v,v)$ \cite[\S 3.7]{MR515141}. 
We have that $\mathcal{X} L = 0$, in particular $\mathcal{X}$ is tangent to $\fdcone$.  In case the distribution function $f$ is a function $f: \mathcal C^+ \to \Reals$, we shall refer to equation \eqref{eq:Vlasov} as the massless Vlasov equation. 

A local coordinate system $(x^a)$ on $\mathcal{M}$ induces natural coordinates $(x^a, v^a)$ on $T\mathcal{M}$, where $v^a = dx^a(v)$. The coordinate form of $\mathcal{X}$ is 
\begin{equation}
\mathcal{X} = v^a \left( \dfrac{\partial}{\partial x^a} -  v^b \Gamma^{c}{}_{ab}  \dfrac{\partial}{\partial v^c}\right),
\end{equation}
where $\Gamma^c{}_{ab}$ is the Christoffel symbol of the metric $g_{ab}$. 

In the Kerr exterior, it is convenient to use the Boyer-Lindquist coordinates $(\xb^a) = (\tb,\rb,\hb,\pb)$ and the corresponding natural coordinates $(x^a, v^a)$. On $\fdcone$, we locally use coordinates $(t,r,\theta,\phi,v^r,v^\theta,v^\phi)$, and treat the quantities $v^t, v_t, v_r, v_\theta, v_\phi$ as functions of these. 
To facilitate the presentation of our main result, we introduce 
\begin{align*}
\EnergyModel[f](t)
= \int_{\Sigma_t} \int_{\fdconex} 
\left( \frac{(r^2+a^2)^2}{\KDelta}v_t^2 +\KDelta v_r^2 +v_\theta^2 +\frac{1}{\sin^2\theta}v_\phi^2 \right)
|f|_2
\diVolModelv \diVolModelx ,
\end{align*}
where 
\begin{align*}
|f|_2=\left|M^2 v_t^2 +v_\theta^2 +\frac{1}{\sin^2\theta}v_\phi^2\right|^2 f,\quad
\diVolModelx=\sin\theta \di\rb\di\hb\di\pb ,\quad
\diVolModelv=\frac{1}{|v_t|}
r^2 \sin\theta \di \vb^r \di \vb^\theta \di \vb^\phi .
\end{align*}
The term $|f|_2$ should be understood as a strengthening of the $f$ by two factors of  $M^2 v_t +v_\phi^2 +\sin^{-2}\theta v_\theta^2$. As explained in Section \ref{sec:morawetz2}, these two factors arise from strengthening the energy by two second-order multiplication symmetries of the Vlasov equation. The volume forms $\diVolModelx$ and $\diVolModelv$ are given here
because they have simple coordinate expressions, although they are not
the naturally induced volume forms defined on $\Sigma_t$ and
$\fdconex$, which are used in the rest of this paper and introduced in Sections  \ref{sec:multsym} and \ref{sec:stressenergy}.

Our main results are
\begin{theorem}[Uniformly bounded energy]
\label{IntroThm:BoundedEnergy}
There are positive constants $C$ and $\bar{\epsilon}$ such that if
$M>0$, $|a|\leq \bar{\epsilon}M$, and $f:\fdcone\rightarrow[0,\infty)$
  is a smooth solution of the Vlasov equation \eqref{eq:Vlasov} in the
  exterior of the Kerr spacetime with parameters $(M,a)$, then, for all $t$ in $\Reals$,

\begin{align}
\EnergyModel[f](t)
\leq C \EnergyModel[f](0) .
\end{align}
\end{theorem}

\begin{theorem}[Morawetz estimate]
\label{IntroTheorems:Morawetz}
There are positive constants $C$, $\bar{\epsilon}$, and $\bar{r}$ and
a function $\chiFar$ which is identically $1$ for $|r-3M|\geq \bar{r}$
and zero otherwise such that if $M>0$, $|a|\leq \bar{\epsilon}M$, and
$f:\fdcone\rightarrow[0,\infty)$ is a smooth solution of the Vlasov equation \eqref{eq:Vlasov} in the exterior of the Kerr spacetime with parameters $(M,a)$, then,
\begin{align}
\int_{-\infty}^\infty\int_{\Sigma_t} \int_{\fdconex}
\left( \left(M\frac{\KDelta^2}{(r^2+a^2)^2}\right)v_r^2 +\chiFar\frac{1}{r}\left(M^2v_t^2 +v_\theta^2 +\frac{1}{\sin^2\theta}v_\phi^2\right) \right)
|f|_2 \diVolModelv \diVolModeltx ,\nonumber\\
\leq C \EnergyModel[f](0) ,
\label{eq:IntroMorawetzCutOff}
\end{align}
where $\diVolModeltx=\di\bar{t}\diVolModelx$.

More precisely, 
\begin{align}
\int_{-\infty}^{\infty} \int_{\Sigma_t} \int_{\fdconex} 
M\frac{\KDelta^2}{(r^2+a^2)^2} v_r^2 |f|_{2} 
+r^5\DiffCurlyRTilde\DiffCurlyRTilde \OpL f
\diVolCone\diVol 
\leq C\EnergyModel[f](0), 
\label{eq:IntroMorawetzCurlyR}
\end{align}
where $\DiffCurlyRTilde$ is given in equation
\eqref{eq:DefnDiffCurlyRTilde} and where $\diVolCone$ and $\diVol$ are
the natural volume forms on $\fdconex$ and $\mathcal{M}$ by the metric
$\metricg$. 
\end{theorem}

The main innovation in this paper is to combine the vector field
technique introduced in \cite{fjs15} for proving dispersive estimate
for the relativistic Vlasov equation with earlier work on dispersion
of fields outside a Kerr black hole, in particular the method of
\cite{AnderssonBlue:KerrWave}, see also \cite{MR3467368}.

The method used in \cite{AnderssonBlue:KerrWave} is a generalization of the vector-field method, which relies on the stress-energy tensor and spacetime symmetries to construct momenta appropriate for the proof of energy estimates and integrated energy estimates. The proof of the non-linear stability of Minkowski space \cite{ChristodoulouKlainerman} is an important application of this method. The vector-field method was recently applied to prove dispersive estimates for the relativistic Vlasov equation as part of a proof of non-linear stability for the massless and massive Vlasov-Nordstr\"om systems on Minkowski space \cite{fjs15} (see also \cite{smu} for the nonrelativistic Vlasov equation). Previous stability results for Minkowski space as a solution
of the Einstein-Vlasov system include the massive \cite{ReinRendall} and massless Vlasov cases
\cite{Dafermos:Vlasov} in spherical symmetry and, recently, the massless case without symmetry \cite{Taylor}. 

Energy bounds and Morawetz estimates, analogous to Theorems
\ref{IntroThm:BoundedEnergy} and \ref{IntroTheorems:Morawetz}
respectively, have already been proved for the wave equation outside a
very slowly rotating Kerr black hole
\cite{DafermosRodnianski:KerrEnergyBound, TataruTohaneanu,
  AnderssonBlue:KerrWave}. Strictly speaking, an energy bound should
be an integral over spacelike hypersurfaces of an integrand that is
quadratic in $v$, rather than of order $6$, as appears in
$\EnergyModel[f]$, but we will consistently ignore this
distinction. Away from (an open set about) $r=3M$, the horizon at
$r=r_+$, and null infinity at $r\rightarrow\infty$, the integrand in
the Morawetz estimate is a bounded multiple of the integrand appearing
in the energy; however, the integral is over all of space-time,
instead of a single spacelike hypersurface. Thus, the Morawetz
estimate implies that the local energy in a fixed $r$ region (away
from $r=r_+$ and $r\rightarrow\infty$ and sufficiently far from
$r=3M$) is integrable in time. Hence, on average, it must decay in
time. Thus, Morawetz estimates are also called integrated local energy
decay estimates. Energy bounds and Morawetz estimates are a useful
tool in proving pointwise estimates, for instance of the form
$\sup_{r\in (r_+,R],(\theta,\phi)\in S^2}|\psi(t,r,\theta,\phi)|
\lesssim t^{-p}$ for some $p$.  
For the wave equation in the subextremal range
$|a|<M$, the entire argument from energy estimates and Morawetz bounds
to pointwise bounds can be found in
\cite{DafermosRodnianskiShlapentokh-Rothman}. 

In the Kerr spacetime, there are null geodesics that can orbit at
fixed $r$, and these are the primary obstacle in proving Morawetz
estimates. Furthermore, for $|a|>0$, the vector field $\partial_t$
ceases to be timelike near $r=r_+$, which prevents the existence of a
conserved, positive energy. Since both the wave equation and massless
Vlasov equation admit solutions that approximate null geodesics for
arbitrary lengths of time, any Morawetz estimate must degenerate on
such solutions. On the orbiting null geodesics, the factor
$\DiffCurlyRTilde$ vanishes, providing sufficient degeneracy; for
$|a|\ll M$, the roots of $\DiffCurlyRTilde$ are all near $r=3M$, which
is why $\DiffCurlyRTilde^2$ in equation \eqref{eq:IntroMorawetzCurlyR} can
be replaced by $\chiFar (M^2v_t^2 v_\theta^2 +\sin^{-2}\theta
v_\phi^2)^2$ in equation \eqref{eq:IntroMorawetzCutOff}. Our analysis is
dependent on the fact that $\DiffCurlyRTilde$ can be expanded purely
in terms of $r$ dependent factors and constants of motion along the
null geodesics. This is a consequence of the remarkable observation of
Carter that, in addition to the geodesic constants of motion arising
from the metric and the two Killing vectors, there is a fourth
constant of motion, called a hidden symmetry \cite{PhysRev.174.1559}.

Steady states for the massive Vlasov
equation in the exterior of a fixed Schwarzschild space-time (where
$a=0$ and representing the exterior of a star or black hole) have been
constructed and used to study accretion disks
\cite{RiosecoSarbach}. The existence of these steady states implies
that no Morawetz estimate, analagous to Theorem
\ref{IntroTheorems:Morawetz}, can hold for the massive Vlasov equation
outside a Schwarzschild black hole.

The formation of black holes for the massive Einstein-Vlasov system has
been studied in
\cite{livingreview,akr11}. For the coupled Einstein-massless Vlasov
system, there are spherically symmetric steady states
\cite{2015arXiv151101290A}. The existence of such solutions suggests, in
contrast to the results in this paper, that there are spherically
symmetric solutions of the Einstein-massless Vlasov system which have a
nonzero, static configuration of massless Vlasov matter outside a
Schwarzschild-like black hole. However such solutions seem to require a
large Vlasov field and cannot be small perturbations of the
Schwarzschild solution.

The paper is organized as follows. Section \ref{sec:prel} contains an introduction to the geometry of the Kerr spacetime, and in particular a discussion on the multiplication symmetries of the field in Section \ref{sec:multsym}, and a presentation of the properties of the stress-energy tensor of the Vlasov equation in Section \ref{sec:stressenergy}. Section \ref{sec:morawetz} contains the proof of the Morawetz estimates. The relevant energies are defined in Section \ref{sec:morawetz1}; Section \ref{sec:morawetz2} introduces the vector field used to perform the estimates; relevant bulk terms are estimated in Section \ref{sec:morawetz3}; and the proof is concluded in Section \ref{sec:morawetz4}.

%%%%%%%%%%%%%%%%%%%%%%%%%%%%%%%%%%%%%%%%%%%%%%%%%%%%%%%%%%%%%%%%%%%%%%%%
\section{Preliminaries} \label{sec:prel}

Throughout, the indices $\ia, \ib, \ic, \dots$ denote integers in $\{0, 1,2,3\}$. Underlined indices $\ua_1, \dots,\ua_k$ are used exclusively to parametrize the set of symmetries used for the calculation, as explained in Section \ref{sec:multsym}. The Einstein summation convention is used throughout the paper.

%%%%%%%%%%%%%%%%%%%%%%%%%%%%%%%%%%%%%%%%%%%%%%%%%%%%%%%%%%%%%%%%%%%%%%%%
\subsection{The Kerr geometry}\label{sec:multsym}
For $M>0$ and $|a|\leq M$, the exterior region of the Kerr space-time
is $(t,r,\omega)\in\Reals\times(r_+,\infty)\times S^2$, where
$r_+=M+\sqrt{M^2-a^2}$ is the larger of the two roots of $\KDelta=0$,
together with the metric given in equation
\eqref{eq:KerrMetricIntrouctureucture}. Typically, we will parameterise $S^2$ be
spherical coordinates $\theta,\phi$. Although this exterior can be
extended as a smooth Lorentzian manifold, it is globally hyperbolic,
with the surfaces of constant $t$ providing a foliation by Cauchy
hypersurfaces.  Thus, there is a well-posed initial-value-problem for
many PDEs, including the Vlasov equation, with initial data posed, for
example, on the hypersurface $t=0$. For $M>0$ and $|a|\leq M$, the
exterior region of the Kerr space-time describes the exterior region
of a rotating black hole.

The vector field $\vecTperp=\pt +\omegaperp\pp$ with
$\omegaperp=2aMr/\KPi$ is orthogonal to the surfaces of constant $t$,
and hence to $\pr,\ph,\pp$. This vector field is not normalised, and,
instead, $\metricg(\vecTperp,\vecTperp)=-\KDelta\KSigma/\KPi$. The
rotation speed of the black hole is
$\omegaH=a/(r_+^2+a^2)$. Independently of $\theta$, one has
$\omegaH=\lim_{r\rightarrow r_+}\omegaperp$.

Many calculations are simplified by working only with the following
form of the inverse Kerr metric:
\begin{align}
\KSigma g^{\ia\ib}
&=\KDelta \pr^\ia\pr^\ib
+\frac{1}{\KDelta}\CurlyR^{\ia\ib} ,
\label{eq:InverseMetricWithCurlyR}
\end{align}
where
\begin{align*}
\KDelta&=r^2-2Mr+a^2,\\
\KSigma&= \Omega^{-2 } = r^2+a^2\cos^2\theta,\\
\CurlyR^{\ia\ib}
&=-(r^2+a^2)^2\pt^\ia\pt^\ib -4aMr\pt^{(\ia}\pp^{\ib)} +(\KDelta-a^2)\pp^\ia\pp^\ib +\KDelta\TensorQ^{\ia\ib},\\
\TensorQ^{\ia\ib}
&=\ph^\ia\ph^\ib +\cot^2\theta\pp^\ia\pp^\ib +a^2\sin^2\theta\pt^\ia\pt^\ib .
\end{align*}

This form of the expression allows us to avoid having to work with
$\KPi=(r^2+a^2)^2-a^2\KDelta\sin\theta$, except when working with
$\vecTperp=\pt +(2aMr/\KPi)\pp$. In fact, except in the volume form
and inside $\TensorQ^{\ia\ib}$, this notation typically allows us to avoid all
$\theta$ dependent factors. It will later be useful to use a conformal factor $\Omega$ defined by 
\begin{equation*}
\Omega^{-2 } = \KSigma.
\end{equation*}
The volume form on the Kerr exterior in Boyer-Lindquist coordinates is
\begin{align*}
\sqrt{|\metricg|}\di t\di r\di\theta\di\phi
&=\KSigma\sin\theta \di t\di r\di\theta\di\phi .
\end{align*}
Since $\KSigma$ is uniformly equivalent to $r^2$, integrals with respect to this volume form are equivalent to those with respect to $\diVolModelx \di t$. 

In the Kerr spacetime, $\pt$ and $\pp$ are Killing vectors and $\TensorQ^{\ia\ib}$ is a conformal Killing tensor. We use the following notation
\begin{align*}
e&= v_\ia \pt^\ia, &
l_z&= v_\ia \pp^\ia, &
q&= v_\ia v_\ib \TensorQ^{\ia\ib} .
\end{align*}
A basis for the multiplicative factors that give symmetries for the null Vlasov equation is
\begin{align*}
\mathbb{S} &= \bigcup_{n=0}^\infty \mathbb{S}_n, &
\mathbb{S}_n&= \{ e^{n_t} l_z^{n_\phi} q^{n_q}: n_t+n_\phi+2n_q=n\} .
\end{align*}
Of particular importance in this analysis is 
\begin{align*}
\mathbb{S}_2&= \{e^2, el_z, l_z^2, q\} =\{\SymOp_\ua\}_{\ua}. 
\end{align*}
Since each element of $\mathbb{S}_2$ is the contraction of a $2$-tensor with $v_\ia v_\ib$, thus, we introduce $\SymOp_\ua^{\ia\ib}$ such that
\begin{align*}
\SymOp_\ua &= \SymOp_\ua^{\ia\ib} v_\ia v_\ib .
\end{align*}

The quantity $\CurlyR^{\ia\ib}v_\ia v_\ib$ plays a crucial role in our
analysis. It can be written as a linear combination of the
$\SymOp_\ua$, with coefficients that are polynomial in $r,M,a$. To
simplify a lot of the calculations, we use the notation $\CurlyR^\ua$
to denote these coefficients and use the Einstein summation convention
in the $\ua$ indices. We also use the notation $\CurlyR$ to denote
$\CurlyR^{\ia\ib}v_\ia v_\ib$. Thus, we have four expressions for the
following quantity
\begin{align}\label{eq:Ra}
\CurlyR
&=\CurlyR^{\ia\ib} v_\ia v_\ib
=\CurlyR^\ua \SymOp_\ua
=\CurlyR^\ua \SymOp_\ua^{\ia\ib}v_\ia v_\ib .
\end{align}
For other quantities that are linear combinations of the $\SymOp_\ua$
(possibly with coefficients that are polynomial or rational functions
of $r,M,a$), we also use the Einstein summation convention in the
$\ua$ variables to expand the quantity. In addition to $\CurlyR$ and
derivatives of rescalings of it, we also use
\begin{align*}
\OpL
&= M^2\GeodesicEnergy^2 +\GeodesicLz^2 +\GeodesicQ ,
\end{align*}
which can be expanded as 
\begin{equation}\label{eq:La}
\OpL^\ua \SymOp_\ua
=\OpL^\ua \SymOp_\ua^{\ia\ib}v_\ia v_\ib.
\end{equation}

One crucial way in which $\CurlyR$ appears in the analysis of whether null geodesics fall into the black hole, asymptote to an orbiting null geodesic, or escape to infinity. Equation \eqref{eq:InverseMetricWithCurlyR} can be contracted with $v_\alpha v_\beta$ to derive an ODE for $\di r/\di \lambda$. One consequence of this is that a null geodesic has a turning point, where $\di r/\di \lambda$ vanishes, only when $\CurlyR=0$. Furthermore, by a standard dynamical systems analysis of one-dimensional systems, there can only be a trajectory remaining at fixed $r$ when $\CurlyR=0=\partial_r\CurlyR$. 

Following \cite{AnderssonBlue:KerrWave}, in the remainder of the paper, it is useful to consider double-indexed collections of vector fields $X^{\ia\ua\ub}$. These are sometimes called $2$-symmetry-strengthened vector fields. The same notation is also used for stress-energy tensors, cf. Section \ref{sec:stressenergy}.

%%%%%%%%%%%%%%%%%%%%%%%%%%%%%%%%%%%%%%%%%%%%%%%%%%%%%%%%%%%%%%%%%%%%%%%%
\subsection{The stress-energy tensor and symmetries of the Vlasov equation}\label{sec:stressenergy}
 %Original version
Throughout this subsection, let $(\mathcal{M},g)$ be a globally
hyperbolic, Lorentzian manifold of dimension $3+1$. Consider the vector bundle
$\mathcal{V}=T\mathcal{M}$, and consider $\fdcone$. 

The volume element on $\fdconex$ induced from the volume element $ \di \mu_{T_x \mathcal{M}} = (-\det \metricg)^{1/2} \di v^0 \wedge\ldots\wedge \di v^3$ is given by the Gelfand-Leray form \cite[Chapter 7]{MR2919697} of $\di \mu_{T_x \mathcal{M}}$ with respect to $-L=-\frac12\metricg(v,v)$, restricted to $\fdconex$. That is $\di \mu_{T_x\mathcal{M}} = \di L \wedge \diVolCone$, where $\di L$ is the exterior derivative of $L$ calculated on $T_x \mathcal{M}$. In local coordinates $(x^a)$ with $a$ taking values $0, \dots, 3$, this takes the (not unique) form 
\begin{align*}
\diVolCone = \sqrt{|\metricg|}\frac{\di v^1 \wedge \di v^2 \wedge \di v^3}{(-v_0)}.
\end{align*}
This can also be found as an appropriate limit of $i_X \left(\di \mu_{T_x\mathcal{M}}\right)$ with $X^\ia =(-\metricg(v,v))^{-1} v^\ia$ on the hypersurface $\{ v: \metricg(v,v)= - m^2,\, \text{and}\,v\, \text{future directed}\}$ as $m\rightarrow 0^+$.

In the particular case of the Kerr spacetime, recall $\sqrt{|\metricg|}=\Sigma\sin\theta$ is uniformly equivalent to $r^2\sin\theta$. For $r$ large, $v_0$ is negative on $\fdconex$. For $r$ near $r_+$, where $\partial_t$ ceases to be timelike, the situation is more complicated. If $v_0<0$, then $(r,\theta,\phi)$ have the standard orientation, and if $v_0>0$, then $(v^r,v^\theta,v^\phi)$ have the reverse orientation. $v_0=0$ only occurs on a set of codimension $1$. Therefore, when integrating $\diVolCone$ over $\fdconex$ using the  orientation induced by $(v^r,v^\theta,v^\phi)$, one always has $|v_0|^{-1}$ in the denominator. As a consequence, integrals with respect to $\diVolCone$ on $\fdconex$ with the orientation induced by a future-directed normal are uniformally equivalent to   integrals over $\fdconex$ computed in the $(v^r,v^\theta,v^\phi)$ coordinates using the measure $\diVolModelv$.

Recall the following (see \cite{livingreview}):
\begin{definition}
\label{def:StressEnergy}
The Vlasov stress-energy
tensor is defined to be 
\begin{align*}
\EMS_{\ia\ib}[f]_x
&= \int_{\fdconex} f(x,v) v_\ia v_\ib \diVolCone.
\end{align*}
\end{definition}
For the remainder of this paper, the term ``stress-energy tensor''
will refer to the Vlasov stress-energy tensor. The Vlasov stress-energy tensor is symmetric,
traceless, and divergence-free for the massless Vlasov equation. If
$f$ is non-negative, the stress-energy tensor satisfies the dominant
energy condition.

Killing tensors play a crucial role in understanding the symmetries of
the Vlasov equation. Recall $K_{\ia_1\ldots\ia_n}$ is a conformal
Killing tensor if, for some $n\in\Naturals$, 
there is a tensor
field $p_{\ia_1\ldots\ia_{n-1}}$ such that
\begin{align*}
K_{\ia_1\ldots\ia_n}&=K_{(\ia_1\ldots\ia_n)}, \\
\nabla_{(\ib}K_{\ia_1\ldots\ia_n)}&=g_{(\ib\ia_1}p_{\ia_2\ldots\ia_n)}.
\end{align*}
 If $K$ is a conformal Killing tensor, then there are several relevant and
well-known consequences. $K_{\ia_1\ldots\ia_n}\dot\gamma^{\ia_1}\ldots\dot\gamma^{\ia_n}$
is constant along any null geodesic $\gamma$. On $T\mathcal{M}$, the function
$(x,v)\mapsto K_{\ia_1\ldots\ia_n}v^{\ia_1}\ldots v^{\ia_n}$ is a
solution of the Vlasov equation \eqref{eq:Vlasov}. Hence, its
restriction to $\fdcone$ satisfies the massless Vlasov
equation. This is a well-known property of the Vlasov fields, which  has already been exploited in \cite[Section 2.8]{fjs15}. From this the following follows by direct calculation. 
\begin{lemma}
If $K_{\ia_1\ldots\ia_n}$ is a conformal Killing tensor, then the map 
\begin{align*}
f(x,v)\mapsto K_{\ia_1\ldots\ia_n}v^{\ia_1}\ldots v^{\ia_n}f(x,v)
\end{align*} 
is a symmetry of the null Vlasov equation, in the sense that if $f(x,v)$ is a solution of the Vlasov equation, then so is 
$K_{\ia_1\ldots\ia_n}v^{\ia_1}\ldots v^{\ia_n}f(x,v)$.
\end{lemma}

\begin{lemma}
Let $n\in\Naturals$, and $\{\SymOp_\ua\}_{\ua}$ be a collection of symmetries for the Vlasov equation of the form $(K_\ua)_{\ia_1\ldots\ia_{N(\ua)}}v^{\ia_1}\ldots v^{\ia_{N(\ua)}}$. The stress-energy tensor defined by
\begin{align*}
\EMS_{\ia\ib\ua_1\ldots\ua_n}[f]_\ptx
&= \int_{\fdconex} \SymOp_{\ua_n}\ldots\SymOp_{\ua_1} f(\ptx,v) v_\ia v_\ib \diVolCone 
\end{align*}
satisfies 
\begin{enumerate}
\item (symmetry) $\EMS_{\ia\ib\ua_1\ldots\ua_n}[f]=\EMS_{\ib\ia\ua_1\ldots\ua_n}[f]$, 
\item (trace-free) $\EMS{}^\ia{}_{\ia\ua_1\ldots\ua_n}[f]=0$, 
and, 
\item (divergence-free) if $f$ is a solution of the Vlasov equation, $\nabla^\ia\EMS_{\ia\ib\ua_1\ldots\ua_n}[f]=0$.

\item (dominant energy condition) Furthermore, suppose $Y^\ia$ is a future-directed causal vector, and
suppose that $X^{\ia\ua_1\ldots\ua_n}$ is such that for any set of
real numbers $\{\sigma_\ua\}_\ua$, the vector
$X^{\ia\ua_1\ldots\ua_n}\sigma_{\ua_1}\ldots\sigma_{\ua_n}$ is a future-directed causal vector. In this
case, if $f$ is nonnegative, then
$\EMS_{\ia\ib\ua_1\ldots\ua_n}[f]X^{\ia\ua_1\ldots\ua_n}Y^{\ib}\geq 0$. 
\end{enumerate}
\end{lemma}
\begin{proof} 
For any sequence of values for $\ua_1,\ldots,\ua_n$, consider the
sequence of concomitants defined by
$\EMS_{\ia\ib\ua_1\ldots\ua_k}[b]=\EMS_{\ia\ib\ua_1\ldots\ua_{k-1}}[\SymOp_{\ua_{k}}
  b]$. Since $\EMS_{\ia\ib}$ is symmetric and is trace-free, the
$\EMS_{\ia\ib\ua_1\ldots\ua_k}$ have the same property by
induction. Similarly, since each $\SymOp_{\ua_k}$ is a symmetry, the
$\EMS_{\ia\ib\ua_1\ldots\ua_k}$ is divergence-free for the Vlasov
equation by induction. 

Suppose the dominant energy condition fails for $\EMS_{\ia\ib\ua_1\ldots\ua_n}$. Thus, there is some smooth $f:\fdcone\rightarrow[0,\infty)$, an $\ptx\in\mathcal{M}$, and $X^{\ia\ua_1\ldots\ua_n}$ and $Y^\ia$ as in the statement of the theorem such that
\begin{align*}
\int_{\fdconex} \SymOp_{\ua_n}\ldots\SymOp_{\ua_1} f(\ptx,v) v_\ia v_\ib X^{\ia\ua_1\ldots\ua_n} Y^\ib \diVolCone 
< 0 .
\end{align*}
Thus, there is a $w\in T_\ptx\mathcal{M}$ such that $\SymOp_{\ua_n}\ldots\SymOp_{\ua_1} f(\ptx,w) w_\ia w_\ib X^{\ia\ua_1\ldots\ua_n} Y^\ib <0$. Let $\sigma_{\ua}$ be the value of $\SymOp_{\ua}$ at $(\ptx,w)$. (Since the $\SymOp_\ua$ are assumed to be multiplicative symmetry operators depending on $(\ptx,v)$, this is possible.) Since the $S_\ua$ are polynomial in the $v^a$, they are continuous in $\fdconex$. Thus, there is an open neighbourhood $W$ of $w$ in $\fdconex$ such that $\SymOp_{\ua_n}\ldots\SymOp_{\ua_1} f(\ptx,v) v_\ia v_\ib X^{\ia\ua_1\ldots\ua_n} Y^\ib <\SymOp_{\ua_n}\ldots\SymOp_{\ua_1} f(\ptx,w) w_\ia w_\ib X^{\ia\ua_1\ldots\ua_n} Y^\ib/2 <0$. Let $\chi$ be a smooth function on $\fdconex$ that is one on an open neighbourhood $W'$ of $w$ and that is supported in $W$. Thus, $f \chi$ is a non-negative function on $\fdconex$ and 
\begin{align}
0
&>\int_{\fdconex}  (\chi(v)f(\ptx,v)) v_\ia v_\ib (X^{\ia\ua_1\ldots\ua_n}\sigma_{\ua_n}\ldots\sigma_{\ua_1}) Y^\ib \diVolCone \nonumber\\
&>\left(\int_{\fdconex}  (\chi(v)f(\ptx,v)) v_\ia v_\ib  \diVolCone\right) (X^{\ia\ua_1\ldots\ua_n}\sigma_{\ua_n}\ldots\sigma_{\ua_1}) Y^\ib
\label{eq:ContradictionForTwoSymmetryStrengthenedEMS}
\end{align}
Since $X^{\ia\ua_1\ldots\ua_n}\sigma_{\ua_1}\ldots\sigma_{\ua_n}$ and $Y^{\ia}$ are timelike and future-directed vector fields, and $\EMS_{\ia\ib}$ satisfies the dominant energy condition, it follows that the final term in inequality \eqref{eq:ContradictionForTwoSymmetryStrengthenedEMS} must be nonnegative, which contradicts inequality \eqref{eq:ContradictionForTwoSymmetryStrengthenedEMS}. Thus, by contradiction, $\EMS_{\ia\ib\ua_1\ldots\ua_n}$ must satisfy the dominant energy condition. 
\end{proof}

One concludes this section by the standard conservation of energies for Vlasov fields. Let$\{\SymOp_\ua\}_{\ua}$ be a collection of symmetries, $X^{a\ua_1\ldots\ua_k}$ a collection of vector fields, and $\Sigma$ be a spacelike hypersurface. The energy of $f$ with respect to the vector $X$ on the hypersurface $\Sigma$ is
\begin{align*}
\GenEnergy{X}[f](\Sigma)
&= \int_\Sigma \EMS_{\ia\ib\ua_1\ldots\ua_k}[f]
X^{\ia\ua_1\ldots\ua_k} 
\diNormal_{\Sigma}^{\ib} ,
\end{align*}

Let now $\Sigma_1, \Sigma_2$ be
hypersurfaces and $R$ be an open set such that $\partial
R=\Sigma_2-\Sigma_1$.  The following lemma states the conservation of energies of Vlasov fields:
\begin{lemma} Let $\Omega$ be a positive function on $\mathcal{M}$, and $q^{\ua_1\ldots\ua_k}$ be a collection of functions on $\mathcal{M}$. The following identity holds:
\begin{align*}
\GenEnergy{X}[f](\Sigma_2)
-\GenEnergy{X}[f](\Sigma_1)
=  \int_R \GenBulk{X,\Omega,q}[f](R)\diVol,
\end{align*}
where
\begin{gather*}
\GenBulk{X,\Omega,q}[f]
= - \frac12 \Omega^2 \EMS_{\ia\ib\ua_1\ldots\ua_k}[f] 
\Lie_{X^{\ua_{1}\ldots\ua_{k}}}(\Omega^{-2} g^{\ia\ib}) 
+\EMS_{\ia\ib\ua_1\ldots\ua_k} \metricg^{\ia\ib}[f]q^{\ua_{1}\ldots\ua_{k}}.
\end{gather*}
\end{lemma}
\begin{remark} Later in Section \ref{sec:morawetz}, to simplify the notations, $\GenBulk{X,\Omega,q}[f]$ is sometimes denoted $\Pi_{X}$, since $\Omega$ and $q$ are clear from context.
\end{remark}
\begin{proof} The proof is a straightforward consequence of the fact that $\EMS_{\ia\ib\ua_1\ldots\ua_k}[f]$ is traceless and divergence free.
\end{proof}

\section{The bounded-energy estimate}\label{sec:morawetz}

The essential parts of the proof are to construct
$2$-symmetry-strengthened vector fields $\vecTBlendBF$ and
$\vecMorawetz$ such that
\begin{subequations}
\label{eq:CoreEstimates}
\begin{align}
\GenEnergy{\vecTBlendBF}&\geq 0, \label{prop1}\\
\GenBulk{\vecMorawetz}  &\geq 0, \label{prop2}\\\
\GenBulk{\vecTBlendBF}  &\lesssim \frac{|a|}{M}\GenBulk{\vecMorawetz} ,\label{prop3}\\\
\GenEnergy{\vecTBlendBF}&\gtrsim |\GenEnergy{\vecMorawetz}| .\label{prop4}\
\end{align}
\end{subequations}
A simple bootstrap argument then shows, for sufficiently small $|a|/M$, that $\GenEnergy{\vecTBlendBF}$ is uniformly bounded by its initial value and that the spacetime integral of $\GenBulk{\vecMorawetz}$ is bounded by a multiple of $\GenEnergy{\vecTBlendBF}$ at any time. 

Of the properties above, the first property \eqref{prop1} is ensured by taking
$\vecTBlendBF$ to be future-directed and causal. The second property \eqref{prop2} is
ensured finding $\vecMorawetz$ (together with a conformal factor and a
collection of auxiliary function $\fnq$) such that (suppressing symmetry indices for simplicity)
\begin{align*}
\Omega^{-2}\GenBulk{\vecMorawetz,\Omega,\fnq}=\left( -\frac12 \Lie_\vecMorawetz(\Omega^{-2}g^{\ia\ib})-\fnq
\Omega^{-2}g^{\ia\ib}\right) \EMS_{\ia\ib} 
\end{align*}
is non-negative. The remaining two properties \eqref{prop3}-\eqref{prop4}
are quantitative, allowing one term to be dominated, rather than
qualitative, merely requiring a term to be signed, and so they are
more complicated. However, if $\vecTBlendBF$ were Killing, then the
associated bulk term would vanish, and the third condition \eqref{prop3} would hold
trivially; since the Kerr exterior has no globally Killing, causal
vector, we instead construct an approximately Killing vector field,
with $|a|/M$ being a measure of the failure of $\vecTBlendBF$ to be
Killing, so that the third condition \eqref{prop3} holds. The fourth condition \eqref{prop4} holds
from the dominant energy condition, as long as $\vecMorawetz$ can be
chosen to have a length bounded by the length of
$\vecTBlendBF$. 

Ideally, one would construct $\vecMorawetz$ and $\vecTBlendBF$ that
are vector fields, but, following \cite{AnderssonBlue:KerrWave}, we
take them to be $2$-symmetry-strengthened vector fields. The Kerr
spacetime has orbiting null geodesics, which we define to be ones
which neither are absorbed through the event horizon nor escape to
null infinity. The projection of such geodesics to $\Sigma_t$ fills an
open set in the Kerr spacetime, but not its tangent space. Because of
the presence of orbiting null geodesics in an open set, it is not
possible to find a vector field $\vecMorawetz$ such that
$\left(\Lie_\vecMorawetz(\Omega^{-2}g^{\ia\ib})-\fnq \Omega^{-2}g^{\ia\ib}\right)\EMS_{\ia\ib}$
is non-negative; however, although \cite{AnderssonBlue:KerrWave}
doesn't use the terminology introduced in this paper, it introduced
$2$-symmetry-strengthened vector fields that, with the energies and
bulk terms for the wave equation, satisfy the four conditions
\eqref{eq:CoreEstimates}. Most of the rest of this paper consists of
constructing these $2$-symmetry-strengthended vector fields and
demonstrating they have the desired properties. In
\cite{AnderssonBlue:KerrWave}, it was important to work with
$2$-symmetry-strengthened vector fields so that the quadratic
stress-energy tensor for the wave equation could be written as a
bilinear quantity. In this paper, it is again convenient to work with
$2$-symmetry-strengthened vector fields, so that we can more easily
define the notion of a causal $2$-symmetry-strengthened vector field. 

The calculations in this paper are significantly simpler than in
\cite{AnderssonBlue:KerrWave}. Both papers rely on properties of null
geodesics and on the fact that, for a geodesic $\gamma$ with the energies and bulk terms defined by
$\GenEnergy{X}[\gamma]=\dot\gamma^\ia X_\ia$ and
$\GenBulk{X}[\gamma]=\nabla_{\ops\ia}X_{\ib\cls}
\dot\gamma^\ia\dot\gamma^\ib$, the estimates
\eqref{prop1}-\eqref{prop4} are valid. The calculations in this paper
are relatively quick, since the behaviour of null geodesics completely
determines the behaviour of solutions to the Vlasov equation. In
contrast, solutions of the wave equation are only accurately modelled
by null geodesics in the high-frequency limit. To treat the wave
equation in \cite{AnderssonBlue:KerrWave}, a significant amount of
additional work is required to show that a similar method can be used
uniformly without a frequency decomposition.

\subsection{The blended energy}\label{sec:morawetz1}
In this subsection, we construct a causal $2$-symmetry-strengthened
vector field. 

\begin{definition}
Let
\begin{align*}
\vecTperp &= \left(\pt+\frac{2aMr}{\KPi} \pp\right)^\ia =  \left(\pt+\omega_{\perp} \pp\right)^\ia  ,\\
\vecTBlend^\ia &=(\pt +\chi\omega_{\mathcal{H}}\pp)^\ia, \\
\vecTBlendBF^{\ia\ua\ub}&= \vecTBlend^\ia  \delta^{\ua\ub} ,
\end{align*}
where $\omega_{\mathcal{H}}=a/(r_+^2+a^2)$ is the rotation speed of the horizon, $\chi=\chi(r)$ is a function that is $1$ for $r<r_\chi$, smoothly decreasing on $r\in[r_\chi,r_\chi+M]$, and identically $0$ for $r>r_\chi+M$, and where $r_\chi$ is chosen sufficiently large. For simplicity, we take $r_\chi=10M$. 
\end{definition}

\begin{lemma}
\label{lem:EnergyBounds}
There is a positive constant $\bar{\epsilon}$ such that if $|a|<\bar{\epsilon}M$, $t\in\Reals$, and $f:\fdcone\rightarrow[0,\infty)$ is continuous, then
\begin{align}
\GenEnergy{\vecTperp}[f](\Sigma_t)
&\simeq \int_{\Sigma_t}\int_{\fdconex} 
\left(\frac{(r^2+a^2)^2}{\KDelta}v_t^2 +\KDelta v_r^2 +\TensorQ^{\ia\ib}v_\ia v_\ib\right) f \diVolCone\diVolSigmat ,
\label{eq:TperpWithQ}\\
&\simeq \int_{\Sigma_t}\int_{\fdconex} 
\left(\frac{(r^2+a^2)^2}{\KDelta}v_t^2 +\KDelta v_r^2 +v_\theta^2+\frac{1}{\sin^2\theta}v_\phi^2 \right) f \diVolCone\diVolSigmat ,
\label{eq:TperpWithhp}\\
\GenEnergy{\vecTperp}[f](\Sigma_t)
&\simeq\GenEnergy{\vecTBlend}[f](\Sigma_t) 
\label{eq:TperpTBlendEquivalence},\\
&\simeq \EnergyModel .
\label{eq:TperpModelEquivalence}
\end{align}
Furthermore, if $f$ is a $C^1$ solution of the Vlasov equation, then
\begin{align}
\sqrt{\det\metricg}\left|-\frac12 \Omega^2 \EMS[f]_{\ia\ib} \Lie_{\vecTBlend}(\Omega^{-2}\metricg^{\ia\ib})\right|
&= \KDelta |\pr\chi| |v_r| |v_\phi| \sin\theta.
\label{eq:TBlendBulk}
\end{align}
\end{lemma}
\begin{proof}
This proof follows the argument of the proof for Lemma 3.1 of
\cite{AnderssonBlue:KerrWave}.

Let $\omegaperp$ denote $2aMr/\KPi$. Since the normal satisfies 
$\diNormal_{\Sigma_t}^\ia=\vecTperp^\ia(\KPi/\KDelta)\di
r\di\theta\di\phi$, and since
$-\metricg_{\ia\ib}\vecTperp^\ia\vecTperp^\ib=\KDelta\KSigma/\KPi$,
the $\vecTperp$ energy is
\begin{align*}
\GenEnergy{\vecTperp}
{}&{}=\int_{\Sigma_{t}}
\EMS_{\ia\ib}\vecTperp^{\ia}\vecTperp^{\ib}
\frac{\KPi}{\KDelta}\di r\di\theta\di\phi
=\int_{\Sigma_{t}}\int_{\fdconex}  f\left(\frac{\KPi}{\KDelta}(\vecTperp^\ia v_\ia)^2
+\frac12\KSigma\metricg^{\ia\ib}v_\ia v_\ib\right)
\diVolCone \di r\di\theta\di\phi .
\end{align*}
The integrand can be expanded as 
\begin{subequations}
\begin{align}
\frac{\KPi}{\KDelta}(\vecTperp^\ia
v_\ia)^2+\frac12\KSigma\metricg^{\ia\ib}v_\ia v_\ib
={}&{} \frac12\left(\KDelta(v_r)^2
+\frac{(r^2+a^2)^2}{\KDelta}(\vecTperp^\ia v_\ia)^2
+\TensorQ^{\ia\ib}v_\ia v_\ib +v_\phi^2\right) 
\label{eq:energyEquivalenceA}\\
{}&{}-\frac{1}{2\KDelta}\left(4aMr-2\omegaperp(r^2+a^2)^2\right)v_t
v_\phi 
\label{eq:energyEquivalenceB}\\
{}&{}+\frac{1}{2\KDelta}\left(-a^2+(r^2+a^2)^2\omegaperp^2\right)v_\phi^2
-a^2\sin^2\theta(\vecTperp^\ia v_\ia)^2 . 
\label{eq:energyEquivalenceC}
\end{align}
\end{subequations}
Since the coefficients $4aMr-2\omegaperp(r^2+a^2)^2$ and
$-a^2+(r^2+a^2)^2\omegaperp^2$ vanish at $r=r_+$, are bounded by
factors that go uniformly to $0$ on bounded sets as $a\rightarrow0$,
and grow as $r\rightarrow\infty$ no faster than $r$ and a constant
respectively, for $|a|$ sufficiently small, the terms in lines
\eqref{eq:energyEquivalenceB}-\eqref{eq:energyEquivalenceC} are dominated
by those on the right-hand side of line
\eqref{eq:energyEquivalenceA}. Thus, the terms on the left and right
side of line \eqref{eq:energyEquivalenceA} are equivalent. This proves
estimate \eqref{eq:TperpWithQ}. Estimate \eqref{eq:TperpWithhp} follows
from the equivalence
\begin{align*}
(\vecTperp^\ia v_\ia)^2 +\TensorQ^{\ia\ib} v_\ia v_\ib
&\simeq (\vecTperp^\ia v_\ia)^2 + v_\theta^2+\frac{1}{\sin^2\theta}v_\phi^2 .
\end{align*}

The $\vecTBlend$ energy can be estimated using the fact that $\vecTperp-\vecTBlend=(\omegaperp-\chi\omegaH)\partial_\phi$ is orthogonal to $\vecTperp$, so
\begin{align*}
\GenEnergy{\vecTperp}-\GenEnergy{\vecTBlend}
{}&{} =\int_{\Sigma_{t}}
(\omegaperp-\chi\omegaH)v_\phi(\vecTperp^\ia v_\ia) \frac{\KPi}{\KDelta} \diVolSigmat .
\end{align*}
The coefficient $\omegaperp-\chi\omegaH$ vanishes linearly at $r=r_+$,
is bounded by a function that goes to zero uniformly as $a\rightarrow
0$, and goes to zero as $r\rightarrow\infty$ like $r^{-4}$, so, by a
simple Cauchy-Schwarz estimate, one finds
$|\GenEnergy{\vecTperp}-\GenEnergy{\vecTBlend}|\lesssim |a|
\GenEnergy{\vecTperp}$, and
$\GenEnergy{\vecTperp}\simeq\GenEnergy{\vecTBlend}$. Finally,
$\GenEnergy{\vecTperp}$ and $\EnergyModel$ are equivalent, since, in
considering the integration on the cone,
$\sqrt{\text{det}\metricg}=\KSigma\sin\theta$ is uniformly equivalent
to $r^2\sin\theta$ for $a$ sufficiently small.

The contraction of the stress-energy tensor with the Lie derivative
can be calculated directly from
$\Lie_{\vecTBlend}(\Omega^{-2}\metricg^{\ia\ib})=-2\KDelta \partial_r^{\ops\ia}\partial_\phi^{\ib\cls}$. 
\end{proof}

\begin{corollary}
There is a positive constant $\bar{\epsilon}$ such that if $|a|\leq\bar{\epsilon}$, $t\in\Reals$, and $f:\fdcone\rightarrow[0,\infty)$ is continuous, then
\begin{align*}
\GenEnergy{\vecTBlendBF}[f](\Sigma(t))
&\simeq \EnergyModel[f](t). 
\end{align*}
\end{corollary}
\begin{proof}
This follows from applying estimates \eqref{eq:TperpWithhp} and \eqref{eq:TperpTBlendEquivalence}, substituting $(M^2 v_t^2 +v_\theta^2 +\csc^2\theta v_\phi^2)^2 f$ for $f$, recognising $\GenEnergy{\vecTBlendBF}[f]$ as being obtained from $\GenEnergy{\vecTBlend}[f]$ substituting $(M^2 v_t^2 +q +l_z^2)^2f$ for $f$, and observing the uniform equivalence of $(M^2 v_t^2 +q +l_z^2)$ and $(M^2 v_t^2 +v_\theta^2 +\csc^2\theta v_\phi^2)$. 
\end{proof}

\subsection{Set-up for radial vector fields}\label{sec:morawetz2}
In this subsection, we define a radial $2$-symmetry-strengthened
vector field, $\vecMorawetz$, in terms of unspecified scalar
functions, which will be chosen in the following subsection. The main
result of this subsection is that the bulk term,
$\GenBulk{\vecMorawetz}$, can be written as a sum of two terms, with
the second involving a square and the first involving a second
derivative. A square is always non-negative. One should expect that the
second-derivative term will be non-negative on orbits, since the orbits
are known to be unstable. In the following subsection, the scalar
functions are chosen so that this second-derivative term is
non-negative everywhere, not just on the orbits.

\begin{definition}
\label{defn:vecMorawetz}
If $\fnMa$ and $\fnMb$ are smooth functions of $r$ and the parameters $M$ and $a$, the Morawetz vector field and the reduced scalar functions are defined to be
\begin{align*}
\vecMorawetz^{\ia\ua\ub}&= -\fnMa \fnMb \OpL^{(\ua} \DiffCurlyRTilde^{\ub)} \pr^\ia , \\
\fnq^{\ua\ub} &= \frac12 (\pr \fnMa)\fnMb \OpL^{(\ua} \DiffCurlyRTilde^{\ub)}, 
\end{align*}
where
\begin{align*}
\DiffCurlyRTilde^\ua&= \pr\left(\frac{\fnMa}{\KDelta}\CurlyR^\ua\right),
\label{eq:DefnDiffCurlyRTilde}\\
\OpL &= \OpL^\ua\SymOp_\ua =M^2 e^2 +l_z^2 + q,\nonumber
\end{align*}
and $\CurlyR^\ua$ and $\OpL$ are defined in equations \eqref{eq:Ra} and \eqref{eq:La}. We also introduce 
\begin{align*}
\DDiffCurlyRTTilde&= \pr \left(\fnMb\frac{\fnMa^{1/2}}{\KDelta^{1/2}}\DiffCurlyRTilde\right).
\end{align*}
\end{definition}

The following lemma is a trivial observation in the current context. This is in contrast with the situation for the wave equation where the rearrangement of the symmetry indices required some calculation and introduced additional terms at the initial and final time, which had to be dominated by the energies. 
\begin{lemma}[Rearrangements]
\begin{align*}
\OpL^{(\ua}\CurlyR^{\ub)}\SymOp_\ua\SymOp_\ub
&=\OpL^{(\ua}\CurlyR^{\ub)}\SymOp_\ua^{\ia\ib}\SymOp_\ub^{\ic\id} v_\ia v_\ib v_\ic v_\id \\
&=\OpL^{\ua}\CurlyR^{\ub}\SymOp_\ua^{\ia\ib}\SymOp_\ub^{\ic\id} v_\ia v_\ib v_\ic v_\id .
\end{align*}
\end{lemma}
\begin{proof}
Apply the definition $\SymOp_\ua=\SymOp_\ua^{\ia\ib} v_\ia v_\ib$ and similarly in $\ub$, and then observe that the contraction in $\ia\ib\ic\id$ is against four copies of $v$, so that it is automatically symmetric in $\ua\ub$. 
\end{proof}

\begin{lemma}
\label{Lemma:BasicMorawetzDeformation}
With $\vecMorawetz$ and $\fnq$ as above, one finds
\begin{align*}
&\GenBulk{\vecMorawetz,\Omega,\fnq}\\
&=\OpL^{\ua}\left(-\fnMa^{1/2}\KDelta^{3/2} \pr\left(\fnMb\frac{\fnMa^{1/2}}{\KDelta^{1/2}}\DiffCurlyRTilde^{\ub}\right)\pr^\ie\pr^\ig 
+\frac12\fnMb\DiffCurlyRTilde^{\ub} \DiffCurlyRTilde^{\uc}\SymOp_\uc^{\ie\ig}\right)\SymOp_\ua^{\ia\ib}\SymOp_\ub^{\ic\id} v_\ia v_\ib v_\ic v_\id v_\ie v_\ig f.
\end{align*}
\end{lemma}
\begin{proof}

Recall
\begin{align*}
\Omega^{-2}\metricg^{\ie\ig}
&=\KDelta\pr^\ie\pr^\ig +\frac{1}{\KDelta}\CurlyR^{\ie\ig}, \\
\Omega^{-2}\GenBulk{\vecMorawetz,\Omega,\fnq}
&=-\frac12 \Lie_{\vecMorawetz^{\ua\ub}}(\Omega^{-2}\metricg^{\ia\ib}) +\Omega^{-2} \fnq^{\ua\ub}\metricg^{\ia\ib} .
\end{align*}
Thus, 
\begin{align*}
\Omega^{-2}\GenBulk{\vecMorawetz,\Omega,\fnq}
&=\Bigg(\frac12\left(\OpL^{(\ua}\fnMa\fnMb\DiffCurlyRTilde^{\ub)}\pr\KDelta-2\KDelta\OpL^{(\ua} \pr\left(\fnMa\fnMb\DiffCurlyRTilde^{\ub)}\right)\right)\pr^\ie\pr^\ig
+\frac12\OpL^{(\ua}\fnMa\fnMb\DiffCurlyRTilde^{\ub)}\pr\left(\frac{\CurlyR^{\ie\ig}}{\KDelta}\right) \\
&\qquad+\frac12(\pr\fnMa)\fnMb\OpL^{(\ua}\DiffCurlyRTilde^{\ub)}\pr^\ie\pr^\ig
+\frac12(\pr\fnMa)\fnMb\OpL^{(\ua}\DiffCurlyRTilde^{\ub)}\frac{1}{\KDelta}\CurlyR^{\ie\ig} \Bigg)\\
&\qquad\SymOp_\ua^{\ia\ib}\SymOp_{\ub}^{\ic\id} v_\ia v_\ib v_\ic v_\id v_\ie v_\ig \\
&=\OpL^\ua\Bigg(-\fnMa^{1/2}\KDelta^{3/2}\pr\left(\fnMb\frac{\fnMa^{1/2}}{\KDelta^{1/2}}\DiffCurlyRTilde^{\ub}\right)\pr^\ie\pr^\ig
+\frac12\fnMb\DiffCurlyRTilde^{\ub}\pr\left(\frac{\fnMa}{\KDelta}\CurlyR^{\ie\ig}\right) \Bigg)\\
&\qquad\SymOp_\ua^{\ia\ib}\SymOp_{\ub}^{\ic\id} v_\ia v_\ib v_\ic v_\id v_\ie v_\ig .
\end{align*}
Substituting
$\DiffCurlyRTilde^{\ie\ig}=\pr\left(\frac{\fnMa}{\KDelta}\CurlyR^{\ie\ig}\right)$
gives the desired result. 
\end{proof}

\subsection{Choosing the weights}\label{sec:morawetz3}
In this subsection, we choose the weights $\fnMa$ and $\fnMb$, so that
$\GenBulk{\vecMorawetz,\Omega,\fnq}$ is non-negative for all $r$. The
choices are the same as those appearing for the wave equation in
\cite{AnderssonBlue:KerrWave}.

Here, we recall how the weight functions $\fnMa$ and $\fnMb$ are chosen, following the explanation in remark 3.8 of \cite{AnderssonBlue:KerrWave}. The goal in choosing the
various weight functions is to obtain nonnegativity for the two terms
in $\GenBulk{\vecMorawetz,\Omega,\fnq}$, namely
$-\fnMa^{1/2}\KDelta^{3/2}
\pr\left(\fnMb\frac{\fnMa^{1/2}}{\KDelta^{1/2}}\DiffCurlyRTilde\right)v_r^2$
and $\frac12\fnMb\DiffCurlyRTilde \DiffCurlyRTilde$. For $|a|\ll M$,
the orbiting null geodesics are near $r=3M$. On orbiting null
geodesics,
$\DiffCurlyRTilde(r;M,a;\GeodesicEnergy,\GeodesicLz,\GeodesicQ)$
vanishes and
$-\DDiffCurlyRTTilde(r;M,a;\GeodesicEnergy,\GeodesicLz,\GeodesicQ)$ is
positive. Thus, the desired non-negativity holds on orbiting null geodesics regardless of the choice of $\fnMa$ and $\fnMb$. The functions $\fnMa$ and $\fnMb$ are chosen so that the non-negativity extends to all other null geodesics. These functions can be chosen so that 
$-\DDiffCurlyRTTilde(r;M,a;\GeodesicEnergy,\GeodesicLz,\GeodesicQ)$
remains positive everywhere and so that $\DiffCurlyRTilde(r;M,a;\GeodesicEnergy,\GeodesicLz,\GeodesicQ)$
vanishes only in a neighbourhood of $r=3M$. 

We have chosen the weights so that the following properties hold:
\begin{enumerate}
\item The definition of $\DiffCurlyRTilde$ in equation
  \eqref{eq:DefnDiffCurlyRTilde} is made so that $\fnMb\DiffCurlyRTilde\pr\left(\frac{\fnMa}{\KDelta}\CurlyR\right)$ takes the form
  $\fnMb\DiffCurlyRTilde^2$ in Lemma \ref{Lemma:BasicMorawetzDeformation}. 
\item $M^2\epsilondtsquared^2$ is the coefficient of
  $\GeodesicEnergy^2$ in
  $\DiffCurlyRTilde(r;M,a;\GeodesicEnergy,\GeodesicLz,\GeodesicQ)$ and
  $\DDiffCurlyRTTilde(r;M,a;\GeodesicEnergy,\GeodesicLz,\GeodesicQ)$. Note
  that $M^2\epsilondtsquared^2$ plays the same role as
  $\epsilon_{\partial_t^2}^2$ in \cite{AnderssonBlue:KerrWave}, where
  the differential symmetry operator $\partial_t^2$ for the wave
  equation plays the role of the multiplicative symmetry
  $\GeodesicEnergy^2$ for the Vlasov equation. The use of a
  dimensionless parameter, $\epsilondtsquared$, in this paper
  clarifies that the small parameter $|a|/M$ can be chosen uniformly
  in $M$ to close the bootstrap argument.
\item $\fnMaa$ is such that, if $\fnMab$ had been equal to $1$, which
  corresponds to $\epsilondtsquared=0$, then the
  coefficient of $M^2\GeodesicEnergy^2$ in $\DiffCurlyRTilde(r;M,a;\GeodesicEnergy,\GeodesicLz,\GeodesicQ)$ would be zero. 
\item $\fnMab$ is such that, if $\epsilondtsquared>0$,
  then the coefficient of
  $M^2\epsilondtsquared\GeodesicEnergy^2$ in $\DiffCurlyRTilde(r;M,a;\GeodesicEnergy,\GeodesicLz,\GeodesicQ)$ is non-negative and a perturbation (in $\epsilondtsquared$) of the coefficient of $\GeodesicQ$.
\item $\fnMba$ is such that, if $\fnMab$ and $\fnMbb$ had both been
  equal to $1$, then the coefficient of $M\GeodesicEnergy\GeodesicLz$
  in
  $\DDiffCurlyRTTilde(r;M,a;\GeodesicEnergy,\GeodesicLz,\GeodesicQ)$
  would vanish.
\item $\fnMbb$ is such that 
\begin{enumerate}
\item  $\DDiffCurlyRTTilde(r;M,a;\GeodesicEnergy,\GeodesicLz,\GeodesicQ)$ is positive
  everywhere, and 
\item $(\fnMa\fnMb
  \DiffCurlyRTilde(r;M,a;\GeodesicEnergy,\GeodesicLz,\GeodesicQ))^2\metricg(\partial_r,\partial_r)\lesssim (M^2\GeodesicEnergy^2+\GeodesicLz^2+\GeodesicQ)^2
  \metricg(\vecTBlend,\vecTBlend)$. \label{condition6.2}
\end{enumerate}
In particular, from the dominant energy condition, condition
\ref{condition6.2} allows us to show that
$\GenEnergy{\vecMorawetz}\lesssim
\GenEnergy{\vecTBlendBF}$. Once the form $\fnMbb=Cr^{-1}$ was
chosen, the factor of $C=1/2$ was chosen so that, when $a=0$ and $\epsilondtsquared=0$, the coefficient of $\GeodesicLz^2+\GeodesicQ$ in
$\DDiffCurlyRTTilde$ is equal to $1$.
\end{enumerate}

The factors $\DiffCurlyRTilde$, $\fnMaa$, $\fnMab$, and $\fnMaa$ are uniquely defined by the above properties. In contrast, the factor $\fnMbb$ is both overdetermined, since we have chosen it to satisfy two conditions that are not a priori obviously compatible, and underdetermined, since it so happens that there are many functions that allow these two conditions to be satisfied.

\begin{definition}
Given a positive value for the parameter $\epsilondtsquared$, we use the following weights to define the Morawetz vector field, 
\begin{align*}
\fnMa&=\fnMaa \fnMab, &
\fnMb&=\fnMba \fnMbb, \\
\fnMaa&= \frac{\KDelta}{(r^2+a^2)^2}, &
\fnMba&= \frac{(r^2+a^2)^4}{3r^2-a^2}, \\
\fnMab&= 1-M^2\epsilondtsquared \frac{\KDelta}{(r^2+a^2)^2} ,&
\fnMbb&= \frac{1}{2r} .
\end{align*}
\end{definition}
The reason for these choices is explained in Remark 3.8 of \cite{AnderssonBlue:KerrWave}. 

In the following lemma, big-$O$ notation is used in the $r$
variable. The notation $f=O(r^{-l})$ means that $f$ is independent of
$v_r$, $\GeodesicEnergy$, $\GeodesicLz$, and $\GeodesicQ$ and that there is
a constant $C$ such that for positive $M$ and sufficiently small $|a|/M$,
uniformly in $r>r_+$, there is the bound $|f(r,M,a)|\leq C r^{-l}$. The
notation $f = g +h O(r^{-l})$ denotes that there is a function
$k=k(r,M,a)$ such that $k=O(r^{-l})$. The notation $f=g +h_1 O(r^{-l_1}) +\ldots
+h_n O(r^{-l_n})$ is defined recursively. 

\begin{lemma}
\label{lem:MorawetzDifferentiated}
There are positive constants $\bar{\epsilon}$, $\epsilondtsquared$, and $C$ such that if $|a|\leq \bar{\epsilon}M$, $0<\epsilondtsquared\leq\bar{\epsilondtsquared}$ and $f:\fdcone\rightarrow[0,\infty)$ is a solution of the Vlasov equation, then
\begin{align}
C \Omega^2 \GenBulk{\vecMorawetz}
\geq&  M\frac{\KDelta^2}{(r^2+a^2)^2} v_r^2 |f|_{2} 
+r^5\DiffCurlyRTilde\DiffCurlyRTilde \OpL f .
\end{align}
and
\begin{align}
\DiffCurlyRTilde
&= -2r^{-4}(r-3M) \OpL_{\epsilondtsquared} \nonumber\\
&\qquad +aM O(r^{-4}) e l_z \nonumber\\
&\qquad +a^2 \left(O(r^{-5})q +O(r^{-5})l_z^2\right) \nonumber\\
&\qquad +M^2\epsilondtsquared\left( a^2O(r^{-5})e^2 +O(r^{-5})q +O(r^{-5})l_z^2\right) .
\label{eq:DiffCurlyRTildeExpanded}
\end{align}
\end{lemma}
\begin{proof}
Direct calculation of $\DiffCurlyRTilde$ with our choices of $\fnMa$ and $\fnMb$ gives
\begin{align*}
\DiffCurlyRTilde
&= -M^2\epsilondtsquared(2(r-3M)r^{-4} +a^2O(r^{-5})) e^2 \\
&\qquad + aMO(r^{-4}) el_z \\
&\qquad -(2(r-3M)r^{-4} +a^2O(r^{-5}) +M^2\epsilondtsquared O(r^{-5}) q \\
&\qquad -(2(r-3M)r^{-4} +a^2O(r^{-5}) +M^2\epsilondtsquared O(r^{-5}) l_z^2 .
\end{align*}
Grouping the terms in orders of $\epsilondtsquared$ and $a$ gives
equation \eqref{eq:DiffCurlyRTildeExpanded}. From this, 
\begin{align*}
-\DDiffCurlyRTTilde
&= M^3\epsilondtsquared(r^{-2} +a^2O(r^{-3}) +M^2\epsilondtsquared O(r^{-3})e^2 \\
&\qquad +aMO(r^{-2}) e l_z \\
&\qquad +M(r^{-2} +a^2O(r^{-3}) +M^2\epsilondtsquared O(r^{-3})) q\\
&\qquad +M(r^{-2} +a^2O(r^{-3}) +M^2\epsilondtsquared O(r^{-3})) l_z^2 .
\end{align*}
In the coefficients of $M^2\epsilondtsquared e^2$, $q$, and $l_z^2$, the $Mr^{-2}$ term dominates the remaining terms for sufficiently small $|a|$ and $\epsilondtsquared$. From the Cauchy-Schwarz inequality, the $el_z$ term is dominated by the $e^2$ and $l_z^2$ terms for sufficiently small $|a|$. Thus, fixing $\epsilondtsquared$ sufficiently small and choosing a constant accordingly, 
\begin{align*}
-\DDiffCurlyRTTilde 
&\geq C M (r^2+a^2)^{-1}(M^2e^2 +q +l_z^2) . 
\end{align*}
Thus, for $|a|$ sufficiently small and $\epsilondtsquared$ as above, 
\begin{align*}
\Omega^2 \GenBulk{\vecMorawetz}
\geq& C M\frac{\KDelta^2}{(r^2+a^2)^2} v_r^2 |f|_{2} 
+\frac{1}{4r}\frac{(r^2+a^2)^4}{3r^2-a^2}\DiffCurlyRTilde\DiffCurlyRTilde \OpL f .
\end{align*}
Thus, there is a new constant $C$, such that
\begin{align*}
C \Omega^2 \GenBulk{\vecMorawetz}
\geq&  M\frac{\KDelta^2}{(r^2+a^2)^2} v_r^2 |f|_{2} 
+r^5\DiffCurlyRTilde\DiffCurlyRTilde \OpL f .
\end{align*}
\end{proof}

\begin{lemma}[Controlling the boundary terms]
\label{lem:ControllingBoundaryTerms}
With $\epsilondtsquared$ as in Lemma \ref{lem:MorawetzDifferentiated}, there is a constant $C$ such that for any $f:\fdcone\rightarrow\Reals$ and $t\in\Reals$, 
\begin{align*}
|\GenEnergy{\vecMorawetz}[f](\Sigma_t)|
\leq C |\GenEnergy{\vecTBlendBF}[f](\Sigma_t)| .
\end{align*}
\end{lemma}
\begin{proof}
This follows from lemma 3.11 of \cite{AnderssonBlue:KerrWave}. Assume
both energies are defined. By direct computation, 
\begin{align*}
\GenEnergy{\vecMorawetz}
{}&{}=-\int_{\Sigma_{t}}
\left(\EMS_{\ia\ib\ua\ub}\vecMorawetz^{\ia\ua\ub}\right)_\ia\vecTperp^\ia\frac{\KPi}{\KDelta} \sin\theta\di r\di\theta\di\phi ,\\
\left|\GenEnergy{\vecMorawetz}\right|
{}&{}\leq C \int_{\Sigma_{t}} \left(|\vecTperp^\ia v_\ia|
|\vecMorawetz^{r\ia\ib}| |\SymOp_\ua\SymOp_\ub|  |v_r|\frac{\KPi}{\KDelta}  
\right)\sin\theta\di r\di\theta\di\phi \\
{}&{}\leq C\int_{\Sigma_{t}}
\left(\frac{\KPi}{\KDelta}|\vecTperp^\ia  v_\ia|_2^2
+\frac{\KPi}{\Delta}\left(\sum_{\ua,\ub}|\vecMorawetz^{r\ua\ub}|^2\right)| v_r|_2^2 
 \right)\sin\theta\di r\di\theta\di\phi .
\end{align*}
Since $\KPi/\KDelta$, $(r^2+a^2)^2/\KDelta$, and $r^4/\KDelta$ are all
uniformly equivalent and since $\sum_{\ua,\ub}\vecMorawetz^{r\ua\ub}$
is bounded by a multiple of $\KDelta r^{-2}$, it follows from estimate
\eqref{eq:TperpWithhp} that $|\GenEnergy{\vecMorawetz}|\leq C
\GenEnergy{\vecTBlendBF}$. Since this bound followed from the
Cauchy-Schwarz inequality, if $\GenEnergy{\vecTBlendBF}$
is finite, then the absolute value of the integrand in
$\GenEnergy{\vecMorawetz}$ is integrable, and
$\GenEnergy{\vecMorawetz}$. If  $\GenEnergy{\vecTBlendBF}$
is infinite, then the desired estimate holds trivially. Thus, the initial
assumption that both energies are finite is redundant. 
\end{proof}

\subsection{Closing the argument}\label{sec:morawetz4}

\begin{proof}[Proof of Theorems \ref{IntroThm:BoundedEnergy} and
    \ref{IntroTheorems:Morawetz}] Let $t_1,t_2\in\Reals$. Initially,
  assume that $f$ restricted to $\pi^{-1}(\Sigma_{t_1})$ has compact
  support in $\pi^{-1}(\Sigma_{t_1})$. By standard results for the
  Vlasov equation, this means $f$ restricted to $\pi^{-1}(\Sigma_{t})$
  has compact support in each $\pi^{-1}(\Sigma_t)$. From integrating
  the result of Lemma \ref{lem:MorawetzDifferentiated}, one finds
\begin{align*}
\GenEnergy{\vecMorawetz}&[f](t_2)-\GenEnergy{\vecMorawetz}[f](t_1)\\
\geq& \int_{t_1}^{t_2} \int_{\Sigma_t} \int_\fdcone 
\left(M\frac{\KDelta^2}{(r^2+a^2)^2} v_r^2 |f|_{2} 
+r^5\DiffCurlyRTilde\DiffCurlyRTilde \OpL f\right)
\diVolCone\diVol .
\end{align*}
Applying Lemma \ref{lem:ControllingBoundaryTerms}, one finds
\begin{align}
\GenEnergy{\vecTBlendBF}&[f](t_2)+\GenEnergy{\vecTBlendBF}[f](t_1)\nonumber\\
\geq& C\int_{t_1}^{t_2} \int_{\Sigma_t} \int_\fdcone 
\left(M\frac{\KDelta^2}{(r^2+a^2)^2} v_r^2 |f|_{2} 
+r^5\DiffCurlyRTilde\DiffCurlyRTilde \OpL f\right)
\diVolCone\diVol .
\label{eq:IntegratedMorawetz}
\end{align}

From multiplying equation \eqref{eq:TBlendBulk} for
$\GenBulk{\vecTBlend}$ by $(M^2 v_t^2 +q +l_z^2)^2$, one obtains the
bulk term for $\GenBulk{\vecTBlendBF}$. Integrating this over
$\bigcup_{t\in[t_1,t_2]}\Sigma_t$, and observing that $|\pr\chi|$ is
compactly supported and that $\omega_{\mathcal{H}}$ vanishes linearly
in $a$, one finds
\begin{align}
\GenEnergy{\vecTBlendBF}&[f](t_2)-\GenEnergy{\vecTBlendBF}[f](t_1)\nonumber\\
&\leq \int_{t_1}^{t_2}\int_{\Sigma_t}\int_{\fdconex} (M^2 v_t^2 +q +l_z^2)^2 \KDelta |\pr\chi| |v_r||v_\phi| f \diVolCone\KSigma^{-1}\diVol \nonumber\\
&\leq \frac{|a|}{M} C 
\int_{t_1}^{t_2} \int_{\Sigma_t} \int_\fdcone 
M\frac{\KDelta^2}{(r^2+a^2)^2} v_r^2 |f|_{2} 
+r^5\DiffCurlyRTilde\DiffCurlyRTilde \OpL f
\diVolCone\KSigma^{-1}\diVol .
\label{eq:GrowthInTBlendBFEnergyByMorawetz}
\end{align}

Combining equations \eqref{eq:IntegratedMorawetz} and \eqref{eq:GrowthInTBlendBFEnergyByMorawetz} and taking $|a|/M$ sufficiently small, one finds that there is a constant $C$ such that 
\begin{align*}
\GenEnergy{\vecTBlendBF}[f](\Sigma_{t_2})
&\leq C \GenEnergy{\vecTBlendBF}[f](\Sigma_{t_1}) .
\end{align*}
Taking $t_2=t$ and $t_1=0$ proves Theorem \ref{IntroThm:BoundedEnergy}
for solutions with compactly supported data. From this, Estimate
\eqref{eq:IntegratedMorawetz}, and taking the limits
$t_2\rightarrow\infty$ with $t_1=0$ and $t_1\rightarrow-\infty$ with
$t_2=0$, one finds equation \eqref{eq:IntroMorawetzCurlyR}. Observing
that $\DiffCurlyRTilde$ grows like $O(r^{-3})(M^2 e^2 q +l_z^2)$ for
large $r$ and has a simple root near $r=3M$ allows us to replace $r^5
\DiffCurlyRTilde\DiffCurlyRTilde \OpL f$ by $r^{-1} \chiFar (Mv_t^2
+v_\theta^2+v_\phi^2)|f|_2$. This proves estimate
\eqref{eq:IntroMorawetzCutOff} and completes the proof of Theorem
\ref{IntroTheorems:Morawetz} for solutions with compactly supported
data. Since the bounds do not depend on the support of the initial
data, by density Theorems \ref{IntroThm:BoundedEnergy} and
\ref{IntroTheorems:Morawetz} hold for all functions for which
$\GenEnergy{\vecTBlendBF}$ is finite. The theorems follow trivially
when this energy is infinite. 
\end{proof}

\section*{Acknowledgements} J\'er\'emie Joudioux is funded in part by the ANR grant ANR-12-BS01-012-01 ``Asymptotic Analysis in General Relativity''.

\printbibliography % <==================================================
%\bibliography{Vlasovb}
%\bibliographystyle{plain}

\end{document}